\documentclass[12pt]{article}
\usepackage{amsmath,amsthm,amsfonts,amsthm,amssymb,color,enumitem}
\usepackage[left=2cm,right=1.5cm,top=1.5cm,bottom=2cm]{geometry}

\usepackage{array,color,colortbl}
\usepackage{tikz}
  \usetikzlibrary{matrix,calc}
  \usetikzlibrary{decorations.markings,decorations.pathreplacing}
  \usetikzlibrary{patterns}

\newtheorem{theorem}{Theorem}

\newtheorem{lemma}[theorem]{Lemma}
\newtheorem{coro}[theorem]{Corollary}
\newtheorem{conjecture}[theorem]{Conjecture}

\theoremstyle{definition}
\newtheorem{defi}[theorem]{Definition}
\newtheorem{exam}[theorem]{Example}

\newcommand{\Z}{\mathbb{Z}}
\newcommand{\R}{\mathbb{R}}

\newcommand{\sym}{\mathcal{S}}

\def\eref#1{{\ensuremath(\ref{#1})}}
\def\sref#1{\S$\ref{#1}$}
\def\lref#1{Lemma~$\ref{#1}$}
\def\tref#1{Theorem~$\ref{#1}$}
\def\egref#1{Example~$\ref{#1}$}
\def\cyref#1{Corollary~$\ref{#1}$}
\def\cjref#1{Conjecture~$\ref{#1}$}

\renewcommand{\=}{\equiv}
\newcommand{\nequiv}{\not\equiv}

\renewcommand{\leq}{\leqslant}
\renewcommand{\ge}{\geqslant}
\renewcommand{\le}{\leqslant}

\tikzset{square matrix/.style={
    matrix of nodes,
    column sep=-\pgflinewidth, row sep=-\pgflinewidth,
    nodes={draw,
      minimum height=15pt,
      anchor=center,
      text width=13pt,
      align=center,
      inner sep=0pt
    },
  },
  square matrix/.default=2cm
}

\colorlet{trans}{blue!50}
\colorlet{trans2}{red!20}

\DeclareMathOperator{\per}{per}

\newcommand{\ev}{^{\rm ev}}
\DeclareMathOperator{\adj}{adj}
\newcommand\mmid{\hspace*{0.5mm}|\hspace*{0.5mm}}

\renewcommand\epsilon{\varepsilon}

\begin{document}

\author{Darcy~Best\thanks{Research supported by Endeavour Postgraduate Scholarship and NSERC CGS-D.}\ \ \ and \ Ian~M.~Wanless\thanks{\texttt{ian.wanless@monash.edu}. Research supported by ARC grant DP150100506.}\\
School of Mathematics\\ Monash University\\ Australia}

\title{Parity of transversals of Latin squares}
\date{}

\maketitle

\begin{abstract}
We introduce a notion of parity for transversals, and use it
to show that in Latin squares of order $2 \bmod 4$, the number
of transversals is a multiple of 4. 
We also demonstrate a number of relationships (mostly congruences
modulo 4) involving $E_1,\dots, E_n$, where $E_i$ is 
the number of diagonals of a given Latin square 
that contain exactly $i$ different symbols. 

Let $A(i\mmid j)$ denote the matrix obtained by deleting row $i$ and
column $j$ from a parent matrix $A$.  Define $t_{ij}$ to be the number
of transversals in $L(i\mmid j)$, for some fixed Latin square $L$. We
show that $t_{ab}\equiv t_{cd}\bmod2$ for all $a,b,c,d$ and $L$. Also,
if $L$ has odd order then the number of transversals of $L$ equals
$t_{ab}$ mod 2. We conjecture that $t_{ac} + t_{bc} + t_{ad} + t_{bd}
\equiv 0 \bmod 4$ for all $a,b,c,d$.

In the course of our investigations we prove several results that
could be of interest in other contexts. For example, we show that the
number of perfect matchings in a $k$-regular bipartite graph on $2n$
vertices is divisible by $4$ when $n$ is odd and $k\equiv0\bmod 4$. We
also show that
$$\per A(a \mmid c)+\per A(b \mmid c)+\per A(a \mmid d)+\per A(b \mmid d) \equiv 0 \bmod 4$$
for all $a,b,c,d$, when $A$ is an integer matrix of odd order with
all row and columns sums equal to $k\equiv2\bmod4$. 

\bigskip
\noindent
Keywords: parity, Latin square, transversal, permanent, 
Latin rectangle, perfect matching, permanental minor, bipartite graph

\noindent
AMS Classifications 05B15, 15A15, 05C70

\end{abstract}

\section{\label{s:intro}Introduction}

A {\em Latin square} is an $n \times n$ matrix consisting of $n$ distinct
symbols where each symbol appears exactly once in each row and each
column. Our Latin squares will have their rows and columns indexed by
$[n]=\{1,2,\dots,n\}$ and will also have their symbols chosen from
$[n]$. Latin squares can then be thought of as a set of {\em entries}
$\{(r,c,s)\} \subset [n]^3$ where each distinct pair of entries agree
in at most one coordinate. The three coordinates of an entry are its
row index, column index and symbol.
A {\em diagonal} of a Latin square is a selection
of $n$ entries, with exactly one entry from each row and each
column. The {\em weight} of a diagonal is the number of distinct symbols on
that diagonal. A diagonal of weight $n$ is called a
\emph{transversal}.  Historically, transversals in Latin squares were
first used as the building blocks of mutually orthogonal Latin
squares (MOLS). They have since garnered a lot of interest on their own
(see~\cite{trans-survey}~for a survey).  A {\em partial transversal of
  length $k$} is a selection of $k$ entries so that no two entries
share the same row, column or symbol.  A partial transversal of length
$k < n$ is not the same thing as a diagonal of weight $k$. While these
objects are related, the distinction is important when counting them.

Over half a century ago, Ryser~\cite{ryser} put forward the following
famous conjecture.

\begin{conjecture}[Ryser's Conjecture]\label{conj:ryser}
 Every Latin square of odd order has a transversal.
\end{conjecture}
 
This conjecture has been shown to be true for $n \leq 9$ by
computation \cite{MMW06}. In 1990, Balasubramanian~\cite{Bal90} showed
that the number of transversals in a Latin square of even order is 
itself even. He claimed that this was a partial proof of a stronger
form of \cjref{conj:ryser}, namely that the number of transversals in
a Latin square of order $n$ should agree with $n$ mod $2$.
Despite \cite{Bal90} attributing this conjecture to \cite{ryser}
it is nowhere to be found in the latter work. It is possible that
Ryser made the conjecture, but we have been unable to find evidence
of this. It is also worth remarking that many Latin squares of odd
order have an even number of transversals, so the stronger form of
\cjref{conj:ryser} is false. However, it does raise the intriguing
possibility of proving existence of objects (in this case, transversals)
by studying congruences satisfied by the number of those objects.
We achieve this on a very modest scale (cf. \egref{eg:existbycong}),
but mostly use it as motivation to unearth what we consider
to be interesting patterns in numbers of transversals and related
quantities.

Akbari and Alipour~\cite{AA04} developed the ideas from
Balasubramanian's result to show that the number of diagonals with
weight $n-1$ is even in every Latin square. We outline and expand on
these ideas in \sref{s:2-mod-4-trans}. We then show that the number of
transversals in a Latin square of order $n \equiv 2 \bmod 4$ is
necessarily a multiple of 4. We show this by exploiting a notion of
parity for transversals. Parity of permutations is, of course, a very
well-known concept. To our knowledge it had not previously been
usefully applied to transversals. However, applying it to the
permutations that define the rows, columns and symbols of a single
Latin square or set of MOLS has previously revealed many insights
\cite{AL15,ABMW14,AT92,Alp17,CW16,DGGL10,FHW18,Gly10,GB12,Jan95,KMOW14,Kot12,SW12,Wan04}. In
particular, each Latin square has a row-parity $\pi_r$, which is the
$\Z_2$ sum of the parities of the permutations that define the rows,
and a column parity $\pi_c$ which is defined similarly for columns
(there is also a symbol parity $\pi_s$, but it is a function of $\pi_r$ and
$\pi_c$, see \cite{DGGL10,Jan95,Wan04}).  We will demonstrate several
new ways to partition Latin squares of certain orders into two types
(independently of their $\pi_r$ and $\pi_c$).

In \sref{s:depleted}, we consider counts of transversals in (not
necessarily square) submatrices of Latin squares.  A {\em Latin array} 
is a matrix of symbols in which no symbol is repeated within any row, or
within any column. A transversal of an $m\times n$ Latin array is a
selection of $\min(m,n)$ entries in which no pair of entries share
their row, column or symbol. Transversals in Latin arrays are
naturally encountered in attempts to find transversals of Latin
squares by induction. They have been the subject of a recent burst of
activity \cite{BN19,BHWWW18,KY20,MPS19} on the question posed in
\cite{AA04} of how many symbols in a Latin array are enough to make a
transversal unavoidable.  We take a different tack, considering
congruences satisfied by the number of transversals in Latin arrays
formed by removing one row and/or one column from a Latin square.

Transversals are diagonals with the maximum possible number of symbols.
In \sref{s:diags} we count diagonals according to how many symbols
they contain and demonstrate several relationships between the
resulting numbers. In doing so we extend on results obtained in
\cite{AA04,Bal90}.

One of the key tools in our results is a matrix function known as
the permanent. Let $M_n(\Z)$ denote the $n\times n$ integer matrices.
The {\em permanent} of a matrix $A=[a_{ij}]$ in $M_n(\Z)$ is defined by
\begin{equation}\label{e:perdef}
\per A=\sum_{\sigma\in\sym_n}\prod_{i=1}^n a_{i\sigma(i)}
\end{equation}
where the sum is over all permutations in the symmetric group $\sym_n$ 
on $[n]$. 
At several points, we use Ryser's formula~\cite{Rys63} to compute the
permanent of a matrix. It states that for $A=[a_{ij}]\in M_n(\Z)$,
\begin{equation}\label{eq:rysers-formula}
  \per A = \sum_{S \subseteq[n]} (-1)^{n-|S|} \prod_{i=1}^{n}\sum_{j \in S}a_{ij}.
\end{equation}
We also make frequent use of the fact that \eref{e:perdef} agrees,
modulo 2, with the definition of the determinant. In a determinant,
some diagonals are given a negative sign, but $-1\equiv1\bmod2$ so
$\per A\equiv\det A\bmod 2$. As a simple example of how this observation 
can be used, we have:

\begin{lemma}\label{lm:even-perm}
  If $A \in M_n(\Z)$ is such that all row sums are even, then
  $\det A$ and $\per A$ are both even.
\end{lemma}

\begin{proof}
  Since the sum of the entries in each row is even, the columns of
  $A$ are linearly dependent over $\Z_2$. Thus, $\det A \equiv 0
  \bmod 2$, from which the claim follows.
\end{proof}

Throughout the paper, $J$ is an all-ones matrix of the appropriate
order and $\Lambda_n^k$ is the set of all $(0,1)$-matrices of order
$n$ which contain exactly $k$ ones in each row and each column.
We also need notation for the {\em conjugates} of a Latin square $L$. 
For each permutation $abc$ in
$\sym_3$ (written in image notation) there is an $abc$-conjugate of $L$.
It is the Latin square obtained by applying the permutation $abc$ to
the three coordinates in the entries of $L$. For example, the $213$-conjugate
of $L$ is the usual matrix transpose of $L$.

\section{\label{s:2-mod-4-trans}Transversals of Latin squares of even order}

In this section, we layout the ideas used first by
Balasubramanian~\cite{Bal90}, then again by Akbari and
Alipour~\cite{AA04} to count the number of transversals in even
ordered Latin squares modulo 2. For consistency with \cite{AA04}, we
will define $E_m=E_m(L)$ to be the number of diagonals in $L$ that
contain exactly $m$ distinct symbols. In particular, $E_n(L)$ is the
number of transversals in $L$ if $L$ has order $n$. The key idea is to
count transversals using inclusion-exclusion.

\begin{defi}\label{def:angle-r}
Let $f(x_1,\dots,x_n)$ be an arbitrary polynomial from $\R^n$ to
$\R$. Then $\langle r \rangle f$ denotes the sum of the values of $f$
at the $\binom{n}{r}$ vectors in $\R^n$ which have $r$ coordinates
equal to $1$ and $n-r$ coordinates equal to $0$.
\end{defi}

The following result is a slight generalisation of both
\cite[Lemma~2]{Bal90} and \cite[Theorem~2.1]{AA04}. 

\begin{lemma}\label{lem:trick}
Let $f(x_1,\dots,x_n)$ be an arbitrary polynomial from $\R^n$ to
$\R$. Then the sum of the coefficients of monomials in $f$ containing
exactly $m$ distinct variables is
$$\sum_{r=0}^m(-1)^{m-r}\binom{n-r}{n-m}\langle r\rangle f.$$
\end{lemma} 

\begin{proof}
Suppose that $\mu$ is a monic monomial containing $t$ distinct variables.
Then
\begin{align*}
\sum_{r=0}^m(-1)^{m-r}\binom{n-r}{n-m}\langle r\rangle \mu
&=\sum_{r=t}^m(-1)^{m-r}\binom{n-r}{n-m}\binom{n-t}{r-t}\\
&=\binom{n-t}{n-m}\sum_{r=t}^m(-1)^{m-r}\binom{m-t}{r-t}\\
&=\begin{cases}
1&\text{if $t=m$,}\\
0&\text{otherwise}.
\end{cases}
\end{align*}
The result follows.
\end{proof}

For any transversal, $\{(r_i,c_i,s_i)\}$, we define three
corresponding permutations of the index set $[n]$ by 
$\sigma_{r}(r_i) = c_i$, $\sigma_{c}(c_i) = s_i$ 
and $\sigma_{s}(s_i) = r_i$. The following result is immediate.

\begin{lemma}\label{lm:parity-equals-0}
Let $T = \{(r_i,c_i,s_i)\}$ be a transversal of a Latin square $L$
with corresponding permutations $\sigma_r, \sigma_c$ and
$\sigma_s$. Then $\sigma_r \circ \sigma_c \circ \sigma_s$ is the identity
permutation.
\end{lemma}

\begin{proof}
For any given symbol $s_i$, we have 
$s_i \xrightarrow{\sigma_s} r_i \xrightarrow{\sigma_r} c_i 
\xrightarrow{\sigma_c} s_i$.
\end{proof}

As an immediate corollary of \lref{lm:parity-equals-0}, we have that
$\epsilon(\sigma_r)+\epsilon(\sigma_c)+\epsilon(\sigma_s) = 0$, where
$\epsilon : \sym_n \rightarrow \Z_2$ is the standard parity homomorphism
on the symmetric group $\sym_n$.
Thus, we can classify transversals into four types: $T^{000}$,
$T^{011}$, $T^{101}$ or $T^{110}$ where the superscript records the
parities of $\sigma_r$, $\sigma_c$ and $\sigma_s$, respectively. We
will use these parities to aid in counting transversals.

\begin{defi} \label{def:det-Latin}
Let $L$ be a Latin square of order $n$. The \emph{parity of a
  transversal} is the parity of the permutation
$\sigma_r$ defined above. We define $E^{\pm}_n(L)$ to be the
number of transversals in $L$ with $\epsilon(\sigma_r)=0$ (\emph{even
transversals}) minus the number of transversals in $L$ with
$\epsilon(\sigma_r)=1$ (\emph{odd transversals}).
\end{defi}

Though the symbols in our Latin square $L$ are normally from $[n]$, we
sometimes need to utilise the corresponding matrix $L[X]$, where each symbol
$i$ is replaced with a variable $x_i$.

\begin{theorem} \label{th:trans-equation}
Let $L$ be a Latin square of order $n$. Then
\begin{equation}\label{eq:trans-incl-excl}
  E_n(L) = \sum_{r=0}^n(-1)^{n-r}\langle r\rangle \per L[X],
\end{equation}
and
\begin{equation}\label{eq:signed-trans-incl-excl}
  E^{\pm}_n(L) = \sum_{r=0}^n(-1)^{n-r}\langle r\rangle \det L[X].
\end{equation}
\end{theorem}

\begin{proof}
We use the $m=n$ case of \lref{lem:trick}.  The terms in $\per L[X]$
which include $n$ distinct variables correspond to the transversals in
$L$.  Similar terms in $\det L[X]$ correspond to the transversals of
$L$ up to sign. Any transversal which has even parity increases the
sum in \eref{eq:signed-trans-incl-excl} by 1 and any odd transversal
decreases the sum by 1.
\end{proof}

The following lemma is adapted from \cite[Lemma~1]{New78}.

\begin{lemma}\label{lem:complement-det}
Let $A = [a_{ij}]$ be a $(0,1)$-matrix of even order. Define $A^* = [b_{ij}]$ by
$$b_{ij} = \begin{cases}
  a_{ij}   & \text{if the } i^{th} \text{ row has an even number of ones,}\\
  1-a_{ij} & \text{otherwise}.
 \end{cases}
$$ 
Then $\det A + \det A^*$ is even.
\end{lemma} 
 
\begin{proof}
Let $\delta$ be a row vector of ones. By permuting rows if necessary,
we may assume that the first $k$ rows of $A$ have odd sum and the
remaining rows have even sum (permuting rows may alter the sign of the
determinant, but this does not matter modulo 2). Thus, we have  
\begin{align*}
\pm\det A^* &= \det\left(
  \begin{array}{c}
  \delta-A_1\\
  \delta-A_2\\  
  \vdots\\
  \delta-A_k\\
  A_{k+1}\\
  \vdots\\
  A_n
  \end{array}\right)
  =
  \det\left(
  \begin{array}{c}
  \delta-A_1\\
  A_1-A_2\\
  \vdots\\
  A_1-A_k\\
  A_{k+1}\\
  \vdots\\
  A_n
  \end{array}\right)
  =
  \det\left(
  \begin{array}{c}
  -A_1\\
  A_1-A_2\\
  \vdots\\
  A_1-A_k\\
  A_{k+1}\\
  \vdots\\
  A_n
  \end{array}\right)
  +
  \det\left(
  \begin{array}{c}
  \delta\\
  A_1-A_2\\
  \vdots\\
  A_1-A_k\\
  A_{k+1}\\
  \vdots\\
  A_n
  \end{array}\right)\!\!\\
&
  =(-1)^{k}\det A +
  \det\left(
  \begin{array}{c}
  \delta\\
  A_1-A_2\\
  \vdots\\
  A_1-A_k\\
  A_{k+1}\\
  \vdots\\
  A_n
  \end{array}\right)
\equiv\det A\bmod 2,
\end{align*}
by \lref{lm:even-perm}.  The result follows.
\end{proof}

Consider the special case of \lref{lem:complement-det} where the row sums of
$A$ are all the same. If all of the row
sums are even, then the result shows nothing interesting. However,
when each row sum is odd, \lref{lem:complement-det} tells us that
$\det A + \det(J-A) \equiv 0 \bmod 2$.  Balasubramanian~\cite{Bal90}
used this result and \eref{eq:trans-incl-excl} to show the following
theorem (actually, \cite{Bal90} showed a generalisation of this
result, which we discuss later). We give a full proof here, as we will
use a similar technique for several of our new results.

\begin{theorem}\label{th:bala}
  If $L$ is a Latin square of even order $n$ then $L$ has an even
  number of transversals.
\end{theorem}

\begin{proof}
Note that $E_n(L) \equiv E^{\pm}_n(L) \bmod 2$. We pair up
complementary terms in \eref{eq:signed-trans-incl-excl}. In other
words, each term of the sum $\langle r \rangle \det L[X]$ is paired
with the unique term in $\langle n-r \rangle \det L[X]$ for which the
indexing zero-one vectors sum to the all-ones vector. For each of
these pairs of terms, we have one of two situations. If $r$ is even,
then $n-r$ is also even and so both determinants are even, by
\lref{lm:even-perm}.  Alternatively, if $r$ is odd, then each row sum
in $L[X]$ is odd, so $\det L[X] + \det(J-L[X]) \equiv 0 \bmod 2$ by
\lref{lem:complement-det}. Thus, each of the $2^{n-1}$ pairs
contributes a multiple of two to the summation in
\eref{eq:signed-trans-incl-excl}. The result follows.
\end{proof}

To proceed, we need a few linear algebraic results. 

\begin{lemma}\label{lem:det-even-k-0-mod-4}
If $A \in \Lambda_n^k$ where both $n$ and $k$ are even, 
then $\det A \equiv 0 \bmod 4$.
\end{lemma}

\begin{proof}
Since $A \in \Lambda_n^k$, we have that $\det A$ is a multiple of
$k\cdot \gcd(n,k)$, by \cite[Theorem~2]{New78}. The desired result
follows, since $k$ and $\gcd(n,k)$ are both even.
\end{proof}

\begin{lemma}\label{lem:sum-of-complement-0-mod-4}
 Let $n \equiv 2 \bmod 4$ and $k \equiv 1 \bmod 2$. If $A \in \Lambda_n^k$, 
 then $$\det A +\det(J-A)\equiv 0 \bmod 4.$$
\end{lemma}
 
\begin{proof}
Since $A \in \Lambda_n^k$, we have that 
\begin{equation}\label{e:compdet}
k\det(J-A) = (-1)^{n-1} (n-k)\det A,
\end{equation} 
by \cite[Lemma~1]{New78}. The result follows by noting
that $n-k \equiv k \bmod 4$ and that $k$ has a multiplicative inverse
modulo $4$.
\end{proof}

We now have the framework to start counting transversals. 
The proof of our next result is very similar to that of \tref{th:bala}.

\begin{theorem}\label{th:det-mult-4}
If $L$ is a Latin square of order $n \equiv 2 \bmod 4$ then 
$E^{\pm}_n(L) \equiv 0 \bmod 4$.
\end{theorem}

\begin{proof}
Again, we pair up complementary terms in
\eref{eq:signed-trans-incl-excl}. That is, each term of the sum
$\langle r \rangle \det L[X]$ is paired with the unique term in
$\langle n-r \rangle \det L[X]$ for which the indexing zero-one vectors 
sum to the all-ones vector. For each of these pairs of terms, we have one of
two situations. If $r$ is even, then $n-r$ is also even and we use
\lref{lem:det-even-k-0-mod-4} twice to show that both terms are a
multiple of four. Alternatively, if $r$ is odd, then we use
\lref{lem:sum-of-complement-0-mod-4} to show that the terms sum to a
multiple of four. The result follows.
\end{proof}

We now use \tref{th:det-mult-4} to show our first main result, which
strengthens \tref{th:bala} for Latin squares of
singly-even order.

\begin{theorem}\label{th:trans-mult-4}
If $L$ is a Latin square of order $n \equiv 2 \bmod 4$ then 
$E_n(L) \equiv 0 \bmod 4$.
\end{theorem}
  
\begin{proof}
Let $T$ be the set of transversals of $L$. We define $w$, $x$, $y$ and $z$ to
be the number of transversals of type $T^{000}$, $T^{011}$, $T^{101}$
and $T^{110}$, respectively.
By definition, we have 
\begin{equation}\label{eq:trans}w+x+y+z = E_n(L)\end{equation}
and 
\begin{equation}\label{eq:d0}w+x-y-z = E^{\pm}_n(L).\end{equation}

Let $L'$ be the $312$-conjugate of $L$. There
is a natural bijection between $T$ and the set of transversals of
$L'$. Each transversal in $L'$ must be of the form
$\{(c_i,s_i,r_i)\}$, where $\{(r_i,c_i,s_i)\} \in T$. The parity of
each transversal in $L'$ depends on $\sigma_c$ for the corresponding 
transversal in $T$, so we have
\begin{equation}\label{eq:d1}
w-x+y-z = E^{\pm}_n(L').
\end{equation}
Similarly, if $L''$ is the $231$-conjugate of $L$, then each
transversal of $L''$ has parity matching that of $\sigma_s$
for the corresponding transversal in $T$. So we have 
\begin{equation}\label{eq:d2}
w-x-y+z = E^{\pm}_n(L'').
\end{equation}
The sum of \eref{eq:trans}, \ref{eq:d0}, \eref{eq:d1} and \eref{eq:d2} gives us 
\begin{equation}\label{eq:sum}
  4w = E_n(L) + E^{\pm}_n(L) + E^{\pm}_n(L') + E^{\pm}_n(L'').
\end{equation}
\tref{th:det-mult-4} applied to $L$, $L'$ and $L''$ tells us that
$E_n(L) \equiv 0 \bmod 4$.
\end{proof}

Based on computation of small squares, it seems that \tref{th:bala}
and \tref{th:trans-mult-4} are the only general modular restrictions
on the number of transversals of a Latin square. By considering sets
of Latin squares that are connected by turning intercalates (that is,
replacing a subsquare
$\left[\begin{smallmatrix}a&b\\b&a\end{smallmatrix}\right]$ with
$\left[\begin{smallmatrix}b&a\\a&b\end{smallmatrix}\right]$), we were
able to find Latin squares that satisfy every other congruence with
small modulus. For example, suppose that $9 \leq n \leq 11$ and $0
\leq k < m \leq 32$. Except where it would violate \tref{th:bala} or
\tref{th:trans-mult-4}, there is some subset of intercalates in these
Latin squares
\[
  \begin{tikzpicture}
    \begin{scope}
      \matrix (A) [square matrix]{
        1 & 2 & 3 & 4 & 5 & 6 & 7 & 8 & 9 \\
        2 & 1 & 4 & 3 & 6 & 5 & 9 & 7 & 8 \\
        3 & 6 & 1 & 8 & 7 & 9 & 5 & 2 & 4 \\
        4 & 3 & 5 & 6 & 9 & 7 & 8 & 1 & 2 \\
        5 & 4 & 2 & 9 & 8 & 1 & 6 & 3 & 7 \\
        6 & 9 & 7 & 5 & 3 & 8 & 2 & 4 & 1 \\
        7 & 8 & 9 & 1 & 2 & 3 & 4 & 5 & 6 \\
        8 & 5 & 6 & 7 & 4 & 2 & 1 & 9 & 3 \\
        9 & 7 & 8 & 2 & 1 & 4 & 3 & 6 & 5 \\
      };
      \draw[ultra thick] (A-1-1.north west) rectangle (A-9-9.south east);
    \end{scope}
    \begin{scope}[xshift=6cm]
      \matrix (A) [square matrix]{
        1 & 2 & 3 & 4 & 5 & 6 & 7 & 8 & 9 & 10 \\
        2 & 1 & 4 & 3 & 6 & 5 & 8 & 7 & 10 & 9 \\
        3 & 6 & 5 & 7 & 2 & 8 & 10 & 9 & 4 & 1 \\
        4 & 5 & 6 & 8 & 7 & 9 & 2 & 10 & 1 & 3 \\
        5 & 8 & 7 & 9 & 1 & 10 & 4 & 3 & 2 & 6 \\
        6 & 4 & 8 & 10 & 9 & 7 & 1 & 2 & 3 & 5 \\
        7 & 3 & 10 & 5 & 8 & 1 & 9 & 4 & 6 & 2 \\
        8 & 7 & 9 & 6 & 10 & 2 & 3 & 1 & 5 & 4 \\
        9 & 10 & 1 & 2 & 3 & 4 & 5 & 6 & 7 & 8 \\
        10 & 9 & 2 & 1 & 4 & 3 & 6 & 5 & 8 & 7 \\
      };
      \draw[ultra thick] (A-1-1.north west) rectangle (A-10-10.south east);
    \end{scope}
  \end{tikzpicture}
\qquad\qquad
  \begin{tikzpicture}
    \begin{scope}[xshift=2cm]
      \matrix (A) [square matrix]{
        1 & 2 & 3 & 4 & 5 & 6 & 7 & 8 & 9 & 10 & 11 \\
        2 & 1 & 4 & 3 & 6 & 5 & 8 & 7 & 11 & 9 & 10 \\
        3 & 8 & 1 & 6 & 7 & 10 & 11 & 9 & 4 & 5 & 2 \\
        4 & 11 & 2 & 8 & 9 & 7 & 5 & 10 & 1 & 3 & 6 \\
        5 & 3 & 6 & 10 & 8 & 9 & 1 & 2 & 7 & 11 & 4 \\
        6 & 4 & 7 & 9 & 10 & 11 & 2 & 3 & 8 & 1 & 5 \\
        7 & 5 & 8 & 11 & 4 & 2 & 10 & 1 & 3 & 6 & 9 \\
        8 & 7 & 9 & 5 & 11 & 1 & 6 & 4 & 10 & 2 & 3 \\
        9 & 10 & 11 & 1 & 2 & 3 & 4 & 5 & 6 & 7 & 8 \\
        10 & 6 & 5 & 7 & 3 & 8 & 9 & 11 & 2 & 4 & 1 \\
        11 & 9 & 10 & 2 & 1 & 4 & 3 & 6 & 5 & 8 & 7 \\
      };
      \draw[ultra thick] (A-1-1.north west) rectangle (A-11-11.south east);
    \end{scope}
  \end{tikzpicture}
\]
that can be turned to give a Latin square of order $n$ with
$k$~mod~$m$ transversals. We also found examples for $12 \leq n \leq
16$ with the same property, but we do not display them here for the
sake of space.  For $n \leq 7$, there are some sporadic values of $k,m
\leq 16$ where no Latin square of order $n$ contains $k$~mod~$m$
transversals. For $n = 8$, there is no Latin square that contains
$22$~mod~$63$ transversals, while there exists a Latin square that
contains $k$~mod~$m$ transversals for all other $0 \leq k < m \leq 64$
that satisfy \tref{th:bala}.  However, we believe that the
restrictions for $n\le8$ are not interesting; they are simply a result
of there being comparatively few Latin squares of these orders.

The proof of \tref{th:trans-mult-4} leads us to the following
interesting property.

\begin{coro} \label{cor:types}
 Let $L$ be a Latin square of order $n \equiv 2 \bmod 4$.  The numbers
 of transversals in $L$ of types $T^{000}$, $T^{110}$, $T^{101}$ and
 $T^{110}$ are all equal modulo $2$.
\end{coro}
 
\begin{proof}
  Define $w$, $x$, $y$ and $z$ as in \tref{th:trans-mult-4}. Adding
  $\eref{eq:trans}$ to $\eref{eq:d0}$ we find that 
  $2w+2x = E_n(L)+E^{\pm}_n(L) \equiv 0 \bmod 4$ (by \tref{th:det-mult-4} and
  \tref{th:trans-mult-4}) which gives us that $w \equiv x \bmod 2$. 
  Similarly, $\eref{eq:trans}+\eref{eq:d1}$ and
  $\eref{eq:trans}+\eref{eq:d2}$ tell us that $w \equiv y \bmod 2$ and
  $w \equiv z \bmod 2$, respectively.
\end{proof}

It is important to remark that \tref{th:trans-mult-4} is less general
in one respect than Balasubramanian's Theorem \cite{Bal90}. Balasubramanian
proved that the number of transversals in any row-Latin square of even
order is even (a {\em row-Latin square} of order $n$ is an $n\times n$
matrix in which each row is a permutation of $[n]$). 
\tref{th:trans-mult-4} does not generalise to row-Latin
squares.  Below we give two row-Latin squares whose number of
transversals is not a multiple of 4. The row-Latin square of order 2
has 2 transversals and the row-Latin square of order 6 has 6
transversals.

\begin{center}
  \begin{tikzpicture}
    \begin{scope}
      \matrix (A) [square matrix]{
        1 & 2 \\
        1 & 2 \\
      };
      \draw[ultra thick] (A-1-1.north west) rectangle (A-2-2.south east);
    \end{scope}
    
    \begin{scope}[xshift=3cm]
      \matrix (A) [square matrix]{
        1 & 3 & 6 & 2 & 5 & 4 \\
        2 & 1 & 5 & 6 & 4 & 3 \\
        3 & 2 & 4 & 1 & 5 & 6 \\
        4 & 2 & 1 & 5 & 6 & 3 \\
        5 & 2 & 3 & 6 & 1 & 4 \\
        6 & 5 & 2 & 3 & 4 & 1 \\
      };
      \draw[ultra thick] (A-1-1.north west) rectangle (A-6-6.south east);
    \end{scope}
  \end{tikzpicture}
\end{center}

Computational evidence suggests the following generalisation 
of \tref{th:trans-mult-4} and \cyref{cor:types}.

\begin{conjecture}\label{conj:even}
 Let $L$ be a Latin square of even order $n$. Let
 $w$, $x$, $y$ and $z$  be the number of transversals in $L$ of types
 $T^{000}$, $T^{011}$, $T^{101}$ and $T^{110}$, respectively. Then
 \begin{enumerate}[label={\rm(\alph*)}]
  \item $E_n(L) \equiv E^{\pm}_n(L) \bmod 4$ and
  \item $w\equiv x \equiv y \equiv z \bmod 2$.
 \end{enumerate}
\end{conjecture}

\cjref{conj:even} is true for $n \equiv 2 \bmod 4$ since
$E_n(L) \equiv E^{\pm}_n(L) \equiv 0 \bmod 4$ 
(by \tref{th:det-mult-4} and \tref{th:trans-mult-4})
and $w\equiv x \equiv y \equiv z \bmod 2$ (by \cyref{cor:types}).
Note that there are many Latin squares of odd order for which
\cjref{conj:even} is not true. 

\begin{lemma}\label{lm:conjecture-a-iff-b}
 The conditions {\rm(a)} and {\rm(b)} 
 in \cjref{conj:even} are equivalent for Latin squares of even order.
\end{lemma}

\begin{proof}
  By \tref{th:bala},
  we know that $E_n(L) \equiv 0 \bmod 2$ when $n$ is even,  showing 
  condition (a) is equivalent to $E_n(L)+E^{\pm}_n(L) \equiv 0
  \bmod 4$. We may use the same idea as the proof of
  \cyref{cor:types} to show the result.
\end{proof}

\section{\label{s:depleted}Transversals of depleted Latin squares}

In this section we give a number of results around the common theme of
transversals of \emph{depleted Latin squares}, that is, matrices formed
by removing a row and/or a column of a Latin square.
This depleted Latin square is a Latin array.

Given an $n \times n$ matrix $A$, we use $\nu(A)$ to denote the
$\Z_2$-nullity of $A$ and we use $A(i \mmid j)$ to denote the $(n-1)
\times (n-1)$ matrix obtained by deleting row $i$ and column $j$ from
$A$. We start by considering the permanent of this submatrix, which is
analogous to the consideration of minors when computing the
determinant.

\begin{theorem}\label{t:perminorsmod2}
  Let $A \in M_n(\Z)$ for $n > 1$. Then
  \begin{itemize}
  \item $\per A(i \mmid j) \equiv 0 \bmod 2$ for all $i,j$ if and only
    if $\nu(A)\ge 2$.
  \item $\per A(i \mmid j) \equiv 1 \bmod 2$ for all $i,j$ if and only
    if $\nu(A)=1$ and all row and column totals of $A$ are even.
  \end{itemize}
\end{theorem}

\begin{proof}
  It suffices to show analogous properties for determinants since the
  determinant and permanent agree modulo $2$. All calculations 
  in this proof will be working over $\Z_2$, and all minors will be
  of order $n-1$.
  
  If $\nu(A)\ge2$, then for all $i,j$ we know that
  $\nu(A(i \mmid j))\ge1$ so $\det A(i \mmid j) \equiv 0\bmod 2$.
  
  If $\nu(A)=0$, then $A$ has an inverse so the adjugate $\adj(A)$ has
  full rank and hence is not a multiple of $J$ (given that
  $n>1$). Hence not all minors of $A$ are equal.
  
  So suppose that $\nu(A)=1$, and hence $\det A=0$. Since $\nu(A)=1$
  there is at least one minor of $A$ that equals $1$.  
  
  If there is any row or column of $A$ with odd sum, then expanding
  the determinant in that row/column shows that $A$ has at least one
  minor which is zero, and hence not all minors are equal.
  
  It remains to treat the case where each row and column sum of $A$ is even.
  It suffices to show $\det  A(1 \mmid 1) \equiv \det  A(2 \mmid 1) \bmod2$. 
  But
  \begin{align*}
    \det A(1 \mmid 1)+\det A(2 \mmid 1)
    &=\det\left(
    \begin{array}{cccc}
      a_{12}+a_{22} & a_{13}+a_{23} & \cdots & a_{1n}+a_{2n}\\
      a_{32}       & a_{33}       & \cdots & a_{3n}\\
      a_{42}       & a_{43}       & \cdots & a_{4n}\\
      \vdots      & \vdots      & \ddots &\vdots\\
      a_{n2}       & a_{n3}       & \cdots & a_{nn}\\
    \end{array}
    \right) 
    \equiv 0 \bmod 2,
  \end{align*}
  since all column sums are even (cf. \lref{lm:even-perm}).
\end{proof}

Since $\nu(A)\ge1$ whenever all row totals are even, we have:

\begin{coro}\label{cy:perminorsmod2}
  Let $A \in M_n(\Z)$ be such that all row and
  column sums are even.  Then 
  $\per A(a \mmid c) \equiv \per A(b \mmid d) \bmod 2$ for all $a,b,c,d$.
\end{coro}

The previous result gave a congruence mod 2. Our next results
involve congruences mod 4.

\begin{theorem}\label{t:perminorsquad}
  Suppose that $n$ is odd and $k\equiv2\bmod4$. If $A\in\Lambda_n^k$,
  then
$$\per A(a \mmid c)+\per A(b \mmid c)+\per A(a \mmid d)+\per A(b \mmid d) \equiv 0 \bmod 4$$
  for any $a,b,c,d$.
\end{theorem}

\begin{proof}
  If $a = b$ (or symmetrically, $c=d$), then
  $2\big(\per A(a \mmid c) + \per A(a\mmid d)\big) \equiv 0 \bmod 4$, by
  \cyref{cy:perminorsmod2}. Hence, it suffices to consider the case when
  $a=c=1$ and $b=d=2$.  Define
  \[B=
  \left(
  \begin{array}{cccc}
    a_{11}+a_{21}+a_{12}+a_{22} & a_{13}+a_{23} & \cdots & a_{1n}+a_{2n}\\
    a_{31}       +a_{32}        & a_{33}        & \cdots & a_{3n}\\
    a_{41}       +a_{42}        & a_{43}        & \cdots & a_{4n}\\
    \vdots                      & \vdots        & \ddots & \vdots\\
    a_{n1}       +a_{n2}        & a_{n3}        & \cdots & a_{nn}\\
  \end{array}
  \right).
  \]
  Note that $B$ has order $n-1$ and that its first
  row and column each sum to $2k\equiv0\bmod4$, while its other rows and columns
  each sum to $k$. Also, by multilinearity of the permanent,
  \begin{align*}
    \per B &=\per\left(
    \begin{array}{cccc}
      a_{12}+a_{22} & a_{13}+a_{23} & \cdots & a_{1n}+a_{2n}\\
      a_{32}        & a_{33}        & \cdots & a_{3n}\\
      a_{42}        & a_{43}        & \cdots & a_{4n}\\
      \vdots        & \vdots        & \ddots & \vdots\\
      a_{n2}        & a_{n3}        & \cdots & a_{nn}\\
    \end{array}
    \right)
+
    \per\left(
    \begin{array}{cccc}
      a_{11}+a_{21} & a_{13}+a_{23} & \cdots & a_{1n}+a_{2n}\\
      a_{31}        & a_{33}        & \cdots & a_{3n}\\
      a_{41}        & a_{43}        & \cdots & a_{4n}\\
      \vdots        & \vdots        & \ddots & \vdots\\
      a_{n1}        & a_{n3}        & \cdots & a_{nn}\\
    \end{array}
    \right)\\
    &=\per A(1 \mmid 1) + \per A(2 \mmid 1)+\per A(1 \mmid 2)+\per A(2 \mmid 2).
  \end{align*}
  Next, apply \eref{eq:rysers-formula} to calculate $\per B$:
  \[
  \per B = \sum_{S \subseteq[n-1]} (-1)^{n-1-|S|} \prod_{i=1}^{n-1}\sum_{j \in S}b_{ij}.
  \]
  Fix a set $S_0$ and consider the terms corresponding to $S_0$ and
  its complement in the outer summation. We have, 
  \begin{align}
    (-1)^{n-1-|S_0|} \prod_{i=1}^{n-1}&\sum_{j \in S_0}b_{ij}
    + (-1)^{|S_0|}\prod_{i=1}^{n-1}\sum_{j \not\in S_0}b_{ij} \nonumber\\
    &=
    (-1)^{|S_0|}
    \left(\prod_{i=1}^{n-1}x_i
    + (2k-x_1) \prod_{i=2}^{n-1}\left(k-x_i\right)\right)\nonumber\\
    &\equiv (-1)^{|S_0|} \left( 2 \prod_{i=1}^{n-1}x_i
    -k\sum_{j=2}^{n-1}\prod_{i\ne j}x_i \right)
    \bmod 4,\label{e:thibit}
  \end{align}
  where
  \[
  x_i=\sum_{j \in S_0}b_{ij}.
  \]
  Now $\sum_i x_i$ is even, so there is an even number of choices of $i$
  for which $x_i$ is even.
  If this number is non-zero, then \eref{e:thibit} is clearly 0 modulo 4.
  So we may assume that every $x_i$ is odd. But then
  \eref{e:thibit} is 0 modulo 4 again, since each term in the sum is odd
  and there is an odd number of summands.
\end{proof}

\begin{theorem}\label{tm:odd-2mod4-complement-permanent}
Let $n \equiv 1 \bmod 2$ and $k \equiv 2 \bmod 4$. If $A\in\Lambda_n^k$,
then $$\per A+2\per(J-A)\equiv 0 \bmod 4.$$
\end{theorem}

\begin{proof}
By inclusion-exclusion,
\[
\per(J-A)=\sum_{i=0}^n(-1)^i(n-i)!\,\tau_i(A)
\equiv \tau_{n-1}(A)-\per A\bmod 2,
\]
where $\tau_i(A)$ is the sum of the permanents of all $i \times i$
submatrices of $A$.
However, $\per A$ is even by \lref{lm:even-perm}, so
$\per A+2\per(J-A)\equiv \per A+2\tau_{n-1}(A)\bmod4$.
Next, define an $n\times n$ matrix $C = [c_{ij}]$ by 
$c_{ij}=\per A(i\mmid j)$.
By \cyref{cy:perminorsmod2} and \tref{t:perminorsquad}, we know that
modulo $4$ each pair of rows of $C$ either agrees in every position
or differs by 2 in every position. Hence,
up to row and column permutations, $C$ has the block form
\[
\left(
\begin{array}{cc}
C_1&C_2\\
C_3&C_4\\
\end{array}
\right)
\]
where each entry in $C_2\cup C_3$ differs from each entry in $C_1\cup
C_4$ by 2 mod 4. Note that some blocks may be vacuous, but one of the
four blocks must have odd dimensions.  Without loss of generality, we
choose it to be $C_1$.  Now, partition $A$ into $4$ blocks
\[
\left(
\begin{array}{cc}
A_1&A_2\\
A_3&A_4\\
\end{array}
\right)
\]
whose dimensions and locations match the corresponding block of $C$.
Define $n_{r}$ to be the total of a row $r$ in block $A_1$.
Next, consider calculating $\per A \bmod 4$ by taking an expansion along
row $r$:
\begin{equation*}\label{e:invper}
\per A = \sum_{j=1}^{n} a_{rj}c_{rj} 
\equiv n_{r}c_{r1}+(k-n_{r})(c_{r1}+2)\equiv 2c_{r1}-2n_{r}
\bmod 4.
\end{equation*}
The answer must be independent of $r$, which means that $n_{r}\bmod 2$
is constant. Analogous statements hold for row totals in each
block. In particular, $A_3$, which has an even number of rows, must contain
an even number of ones. But then $A_1$ must also contain an even number of
ones, given that the column totals of $A$ are even. It follows that 
$n_r$ must be even, so $\per A\equiv 2c_{r1}\bmod 4$. Now,
\[
\tau_{n-1}(A)=\sum_{i,j}c_{ij}\equiv n^2c_{r1}\equiv c_{r1} \bmod 2.
\]
So $\per A+2\tau_{n-1}(A)\equiv4c_{r1}\equiv0\bmod4$ and we are done.
\end{proof}

In our next major result, we show that a stronger form of
\lref{lm:even-perm} can be obtained under some circumstances.

\begin{theorem}\label{th:odd-reg-01-per-mult-4}
Let $A \in M_n(\Z)$ where $n$ is odd. If all row sums are multiples 
of $4$ and all column sums are even, then $\per A \equiv 0\bmod 4$.
\end{theorem}
  
\begin{proof}
  We compute the permanent via \eref{eq:rysers-formula}:
  \begin{equation*}
    \per A = \sum_{S \subseteq[n]} (-1)^{n-|S|} \prod_{i=1}^{n}\sum_{j \in S}a_{ij}.
  \end{equation*}
  Let $r_i$ be the sum of row $i$ and $c_j$ be the sum of column
  $j$ of $A$. Fix a set $S=S_0$ of odd cardinality 
  and consider the contribution from $S_0$
  and its complement. We have,
  \begin{align}
    \prod_{i=1}^{n}\sum_{j \in S_0}a_{ij} - \prod_{i=1}^{n}\sum_{j \not\in S_0}a_{ij} 
    &= 
    \prod_{i=1}^{n}\sum_{j \in S_0}a_{ij} - \prod_{i=1}^{n}\left(r_i-\sum_{j \in S_0}a_{ij}\right)\nonumber\\
    &\equiv 2 \prod_{i=1}^{n}\sum_{j \in S_0}a_{ij}\bmod 4.\label{eq:twice-product}
  \end{align} 
  Since each column of $A$ has an even total,
  \begin{equation}\label{eq:col-sum}
    \sum_{i=1}^{n}\sum_{j \in S_0} a_{ij} = \sum_{j\in S_0}c_j \equiv 0 \bmod 2.
  \end{equation}
  Since $n$ is odd, $\sum_{j \in S_0} a_{ij}$ must be even for at
  least one value of $i$ in \eref{eq:col-sum}. Thus, the product of
  the partial row sums must be even and \eref{eq:twice-product} must
  be a multiple of 4. Summing over $S_0$, the result follows.
\end{proof}

\begin{coro}\label{cor:odd-n-lambda-4k-mult-4}
  Let $A \in \Lambda^{4k}_n$ for integers $k,n$ with $n$ odd. Then $\per A
  \equiv 0 \bmod 4$.
\end{coro}

It is well-known that perfect matchings in bipartite graphs can be
counted using the permanent of the bi-adjacency matrix of the
graph. \cyref{cor:odd-n-lambda-4k-mult-4} says that the number of
perfect matchings will be a multiple of $4$ in any $4k$-regular
bipartite graph with an odd number of vertices in each class of the
bipartition. Indeed, \tref{th:odd-reg-01-per-mult-4} says that the
same conclusion can be reached under weaker hypotheses. It suffices
for all vertices in one class to have even degree, and all
vertices in the other class to have their degree divisible by 4.

\medskip

We define $t_{ij}(L)$ to be the number of transversals in the Latin
array formed by deleting the $i$th row and $j$th column of $L$. When
clear from context, the shorthand $t_{ij}$ is used.

\begin{theorem}\label{th:same-delta-row}
  Let $L$ be a row-Latin square of order $n$. Then for all
  $a,b,c$, $$t_{ab} \equiv t_{ac} \bmod 2.$$
\end{theorem}
 
\begin{proof}
Without loss of generality, we may assume that $n\ge2$, $a = 1$, $b=1$, $c=2$. 
Let $L[X]=[x_{ij}]$ and define
  $$L[X]' =
  \left(\begin{array}{ccccc}
    1      & 1      & 0      & \cdots & 0 \\
    x_{21} & x_{22} & x_{23} & \cdots & x_{2n} \\
    \vdots & \vdots & \vdots & \ddots & \vdots \\
    x_{n1} & x_{n2} & x_{n3} & \cdots & x_{nn}
  \end{array}\right).
  $$
Then $t_{ab} + t_{ac}$ is the number of terms in $\per L[X]'$ which have exactly
$n-1$ symbols. Thus, by \lref{lem:trick}, 
$$t_{ab} + t_{ac} = \sum_{r=0}^{n-1}(-1)^{n-1-r}(n-r)\langle r\rangle \per L[X]'
  \equiv \sum_{r=1}^{n-1}(n-r)\langle r\rangle \det L[X]' \bmod 2.$$
    
If $n$ is odd, then we have two subcases. If $r$ is even, then
$\langle r\rangle\det L[X]'$ is even, by \lref{lm:even-perm}. If $r$
is odd, then $n-r$ is even, and so each term in the summation is even.
   
If $n$ is even, then we use a trick similar to \tref{th:det-mult-4} by
pairing up complementary terms. Our result follows from
\lref{lm:even-perm} and \lref{lem:complement-det} when 
$r$ is even and $r$ is odd, respectively.
\end{proof}

This immediately gives us a surprisingly simple result which 
lays the groundwork for the patterns found in the remainder of the
section.

\begin{coro}\label{cy:same-delta}
  Let $L$ be a Latin square of order $n$. Then for all $a,b,c,d$, 
  $$t_{ab} \equiv t_{cd} \bmod 2.$$
\end{coro}

\begin{proof}
  Since $L$ is a row-Latin square, $t_{ab} \equiv t_{ad} \bmod 2$ by
  \tref{th:same-delta-row}. Moreover, since the transpose of $L$ is a row-Latin square, 
  $t_{ad} \equiv t_{cd} \bmod 2$.
\end{proof}

This simple observation leads to several patterns relating to deleting
a row and a column of a Latin square.

\begin{coro}\label{cor:row-Latin-even}
 Let $R$ be an $(n-1)\times n$ row-Latin rectangle, where $n$ is
 even. Then the number of transversals in $R$ is even.
\end{coro}

\begin{proof}
  Let $L$ be some row-Latin square formed by adding one row to $R$. By
  definition, the number of transversals in $R$ is
  $$t_{n1}(L) + t_{n2}(L) + \cdots + t_{nn}(L).$$ 
  Each of these terms is congruent modulo 2 (by
  \tref{th:same-delta-row}) and $n$ is even. 
\end{proof}

Each $(n-1)\times n$ Latin rectangle $R$ has a unique completion to a
Latin square $L$, and each transversal of $R$ corresponds to a
so-called {\em near transversal} of $L$.  \cyref{cor:row-Latin-even}
does not generalise to odd orders, as there are some rectangles that
have an even number of transversals and other rectangles that have an
odd number of transversals. If any row is removed from
the Cayley table of a cyclic group of odd order, the resulting Latin
rectangle has an odd number of transversals.  This can been seen by
combining two well-known features of the cyclic group tables of odd
order. Firstly each near transversal extends to a (unique)
transversal, and secondly there are an odd number of transversals.

We define $N_r=N_r(L)$ to be the number of diagonals of weight $n-1$
in $L$ where the symbol that appears in row $r$ also appears in
another row of the diagonal. The
following two results follow directly from the definition of
$t_{ij}$.

\begin{lemma}\label{lm:delta-along-row}
 Let $L$ be a Latin square of order $n$. Then for any row $r$, 
 $$\sum_{c=1}^{n}t_{rc} = E_n+N_r.$$
\end{lemma}
 
\begin{proof}
  Each transversal in the matrix formed by deleting row $r$ and column
  $c$ extends to either a transversal of $L$ or a diagonal of weight
  $n-1$ depending on which symbol is in the cell $(r,c)$.
\end{proof}

\begin{lemma}\label{lm:delta-along-everyone}
 Let $L$ be a Latin square of order $n$. Then 
$$\sum_{r=1}^{n}\sum_{c=1}^{n}t_{rc} = n E_n + 2 E_{n-1}.$$
\end{lemma}

 \begin{proof}
  Across the whole summation, each transversal of $L$ is counted $n$
  times (once for each entry in the transversal) and each diagonal
  with weight $n-1$ is counted twice (once for each entry containing
  the duplicated symbol).
 \end{proof}

Our next main result has a curious feature, which we explain after proving
the result.

\begin{theorem}\label{th:delta-equals-trans}
 Let $L$ be a Latin square of odd order $n$. Then
 for any $r$ and $c$, $$t_{rc} \equiv E_n \bmod 2.$$
\end{theorem}
 
\begin{proof}
  Since $n$ is odd, \cyref{cy:same-delta} ensures that
$$t_{rc} \equiv \sum_{i=1}^{n}\sum_{j=1}^{n} t_{ij} \bmod 2.$$
  Then \lref{lm:delta-along-everyone} gives 
$t_{rc} \equiv n E_n + 2E_{n-1} \equiv E_n \bmod 2$, as desired.
\end{proof}

\begin{coro}\label{cor:n-r-even}
Let $L$ be a Latin square of order $n$. Then $N_r$ is even
for all rows~$r$.
\end{coro}
 
\begin{proof}
  Each term in $\sum_{c=1}^{n}t_{rc}$ is the same modulo 2.  
  If $n$ is even, this sum is even, whereas if $n$ is odd,
  the sum is equivalent to $t_{r1}$ modulo 2. In either case, the sum
  is equivalent to $E_n$ modulo 2, by \tref{th:bala} and
  \tref{th:delta-equals-trans}, respectively. The result now follows
  from \lref{lm:delta-along-row}.
\end{proof}

An interesting feature of \tref{th:delta-equals-trans} lies in the
fact that a transversal of $L$ can be inferred without locating
one. In each of the other previous results, the number of diagonals
with specific properties is of a similar form: congruent to 0 modulo
$m$ for some $m$. Congruences like this cannot be used to show
existence of transversals. However, \tref{th:delta-equals-trans} gives
a slightly different approach. In particular, if $t_{rc} \equiv 1
\bmod 2$ for any row and column, then there must exist a transversal
in $L$ even if none go through the cell $(r,c)$.

\begin{exam}\label{eg:existbycong}
  Consider $L_5$:
  \begin{center}
    \begin{tikzpicture}
      \begin{scope}[xshift=3cm]
        \matrix (A) [square matrix]{
          1 & 2 & 3 & 4 & 5 \\
          |[fill=trans]|2 & 1 & 4 & 5 & 3 \\
          3 & 4 & 5 & 1 & 2 \\
          4 & 5 & 2 & 3 & 1 \\
          5 & 3 & 1 & 2 & 4 \\
        };
        \draw[ultra thick] (A-1-1.north west) rectangle (A-5-5.south east);
      \end{scope}
    \end{tikzpicture}
  \end{center}
  
Every transversal in $L_5$ goes through the shaded entry. 
In particular, there are no transversals including the entry in the
top left corner.
However, the main diagonal is
the sole transversal in $L_5(1\mmid 1)$, so $t_{11} = 1$. Thus, we can
use \tref{th:delta-equals-trans} to deduce that at least one
transversal exists in $L_5$ without finding such a transversal.
\end{exam}

We finish the discussion of $t_{ij}$ with a rather curious pattern
found for small orders. It is very much in the spirit of
\tref{t:perminorsquad} (but does not seem to follow directly from it).

\begin{conjecture}\label{conj:4-entries-add-up-to-4}
 Let $L$ be a Latin square of order $n$. Then 
 $t_{ac} + t_{bc} + t_{ad} + t_{bd} \equiv 0 \bmod 4$ for all $a,b,c,d$.
\end{conjecture}

In light of \cyref{cy:same-delta}, \cjref{conj:4-entries-add-up-to-4}
implies a very specific structure for the matrix $[t_{ij}]$. 
Each pair of rows either agrees modulo 4 or differs in every column
by 2 modulo 4. A similar observation holds for columns.

\section{\label{s:diags}Counting diagonals by their number of symbols}

In this section, we look at relationships between the $E_i=E_i(L)$,
that is, the counts of diagonals of $L$ according to how many symbols
they contain.  We will also be interested in $R_i=R_i(L)$ which we
define to be shorthand for $\langle i \rangle \per L[X]$. 
Note that $R_0=0$. The $R_i$ are related to the $E_i$ by
\begin{equation}\label{eq:e-m}
E_m = \sum_{r=1}^{m} (-1)^{m-r} \binom{n-r}{n-m}R_r,
\end{equation}
where $n$ is the order of $L$.
This relationship was given explicitly in \cite{AA04} and can easily be 
derived from \lref{lem:trick}.

In several proofs we will encounter $d_n$, the number of {\em derangements}
in $\sym_n$. From the well-known recurrence $d_n=nd_{n-1}+(-1)^n$,
we learn that 
\begin{equation}\label{e:deranged}
d_n\equiv1\text{ mod $4$ when $n$ is even.}
\end{equation}

The following proposition is a list of identities which are either
immediate from the definition of a Latin square or are proved in \cite{Bal90}.

\begin{lemma}\label{lm:identities} 
Let $L$ be a Latin square of order $n$.
 \begin{enumerate}[label=\rm({\alph*})]
  \item $R_1 = n$,
  \item $R_{n-1} = nd_n$,
  \item $R_n = n!$,
  \item $R_{2i}$ is even for each integer $i$,
  \item $R_{i} + R_{n-i}$ is even if $n$ is even, and
  \item $R_{n/2}$ is even if $n$ is even.
 \end{enumerate}
\end{lemma}

Balasubramanian~\cite{Bal90} used (d) and (e) to show \tref{th:bala},
while Akbari and Alipour~\cite{AA04} showed the following two results.

\begin{theorem}\label{th:E_n-3}
If $L$ is a Latin square of order $n\=2\bmod4$ then $E_{n-3}$ is even.
\end{theorem}

\begin{theorem}\label{th:E_n-1}
If $L$ is a Latin square of order $n$ then $E_{n-1}$ is even.
\end{theorem}

We start with patterns in Latin squares of odd order. 
We have two direct corollaries of earlier results.

\begin{coro}\label{cor:odd-ls-mult-4}
  If $L$ is a Latin square of odd order then 
  $R_{4k} \equiv 0\bmod 4$ for each integer $k$.
\end{coro}
  
\begin{proof}
  Simply apply \cyref{cor:odd-n-lambda-4k-mult-4} to each 
  matrix in the sum that defines $R_{4k}$.
\end{proof}

\begin{coro}\label{cor:odd-ls-2-mod-4s}
If $L$ is a Latin square of odd order $n$ and $k\equiv2\bmod4$, 
then $$R_k + 2R_{n-k} \equiv 0 \bmod 4.$$
\end{coro}

\begin{proof}
Apply \tref{tm:odd-2mod4-complement-permanent} to each 
of the complementary pairs in $R_{k}$ and $R_{n-k}$.
\end{proof}

In addition, we have: 

\begin{theorem}\label{th:Eioddn}
If $L$ is a Latin square of odd order $n$ then
$E_i$ is even whenever $i$ is even.
\end{theorem}

\begin{proof}
By \eref{eq:e-m} we have
\[
E_i\=\sum_{j=1}^i\binom{n-j}{n-i}R_j
\bmod 2.
\]
Now, $R_j$ is even for even $j$, by \lref{lm:identities}(d). When $j$ is odd,
$\binom{n-j}{n-i}$ is even by Lucas' Theorem, given that $n-j$ is even and
$n-i$ is odd. The result follows.
\end{proof}

We also have the following strengthening of \tref{th:E_n-1} for odd orders:

\begin{theorem}\label{th:En-1}
If $L$ is a Latin square of odd order $n$ then
$E_{n-1} \equiv 0 \bmod 4$.
\end{theorem}

\begin{proof}
  We compute $E_{n-1}$ utilising \eref{eq:e-m}. We pair up the
  complementary terms in this summation, $(n-r)R_r - rR_{n-r}$. 
  Within each of these pairs, we assume that $r$ is
  even, by replacing $r$ by $n-r$ if necessary. 
  We examine two cases. First, if $r \equiv 0 \bmod
  4$, then the second term vanishes modulo 4 and $(n-r)R_r \equiv 0
  \bmod 4$ by \cyref{cor:odd-ls-mult-4}. Alternatively, if $r \equiv 2
  \bmod 4$ then $R_r$ is even, by \lref{lm:identities}(d), so 
  $(n-r)R_r \equiv  R_r \bmod 4$. Thus, 
  $(n-r)R_r - rR_{n-r} \equiv R_r + 2R_{n-r} \bmod
  4$. We may now use \cyref{cor:odd-ls-2-mod-4s}. Each pair of
  complementary terms sums to a multiple of four, so the result
  follows.
\end{proof}

We now shift our attention to Latin squares of even order, where the
results are based on the global relationship between the different
$R_i$ values in contrast with the local nature of
\cyref{cor:odd-ls-mult-4} and \cyref{cor:odd-ls-2-mod-4s}.

\begin{theorem}\label{t:oddevenE}
If $L$ is a Latin square of even order $n>2$ then 
$$E_1 + E_3 + \cdots + E_{n-1} \equiv E_2 + E_4 + \cdots + E_n \equiv n \bmod 4.$$
\end{theorem}

\begin{proof}By \eref{eq:e-m},
\begin{align*}
  \sum_{m=1}^{n/2} E_{2m-1} &= \sum_{m=1}^{n/2} \sum_{r=1}^{2m-1} (-1)^{2m-1-r}\binom{n-r}{n-2m+1} R_r\\ 
&= \sum_{r=1}^{n-1} (-1)^{n-r-1}R_r\sum_{s=1}^{\lceil(n-r)/2\rceil} \binom{n-r}{2s-1} \\
  &= \sum_{r=1}^{n-1} (-2)^{n-r-1} R_r \equiv R_{n-1} - 2R_{n-2} \equiv n d_n - 0 \equiv n \bmod 4,
\end{align*} 
by \lref{lm:identities} and \eref{e:deranged}.
If $n\ge4$ then $\sum_{i=1}^{n} E_i = n! \equiv 0 \bmod 4$, 
and the second congruence follows.
\end{proof}

\begin{coro}\label{cy:evendiag}
Every Latin square has an even number of diagonals that contain an even number
of symbols.
\end{coro}

\begin{proof}
The order $2$ case is trivial and \tref{t:oddevenE} takes care of
all larger even orders.
The odd case is immediate from \tref{th:Eioddn}.
\end{proof}

\begin{coro}
Every Latin square of order $n>1$ has an even number of diagonals that 
contain an odd number of symbols.
\end{coro}

\begin{proof}
There are $n!\=0\bmod2$ diagonals,
so the result follows from \cyref{cy:evendiag}.
\end{proof}

The {\em even permanent} $\per\ev$ is defined as the sum of the
products of the entries on the even diagonals of a matrix. In other
words, it has the same definition as \eref{e:perdef} except that the
sum is taken over the alternating group rather than the symmetric
group. Let $R_i\ev$ be defined the same as $R_i$, but using $\per\ev$
in place of $\per$. Similarly, let $E_i\ev$ be the number of {\em
  even} diagonals with exactly $i$ different symbols on them. We
considered even permanents as one possible approach to
\cjref{conj:even}. While that effort was unsuccessful, we did manage
to prove this weak analogue of \tref{t:oddevenE}:

\begin{theorem}\label{tm:evper}
If $L$ is a Latin square of even order $n>2$ then 
$$E_3\ev + E_5\ev + \cdots + E_{n-1}\ev 
\equiv E_1\ev+ E_2\ev + E_4\ev + E_6\ev + \cdots + E_n\ev \equiv 0 \mod 2.$$
\end{theorem}

\begin{proof}
  Similar to the proof of \tref{t:oddevenE}, we have that 
  \[
    \sum_{m=1}^{n/2} E_{2m-1}\ev \equiv R_{n-1}\ev \mod 2.
  \]
  Hence
  \begin{align}\label{e:evenpereq}
    \sum_{m=3}^{n/2} E_{2m-1}\ev \equiv E_1\ev+R_{n-1}\ev \equiv R_1\ev+R_{n-1}\ev
    \mod 2.
  \end{align}
  Let $P_0$ and $P_1$ be permutation matrices corresponding to arbitrary
  even and odd permutations, respectively.
  Then $\per\ev P_0=1$ and $\per\ev P_1=0$. Also
  $\per\ev(J-P_0)=a$ and $\per\ev(J-P_1)=b$, where $a+b=\per(J-P_0)=d_n$,
  and $a-b=\det(J-P_0)=1-n$, by \eref{e:compdet}.
  Thus $a=(d_n+1-n)/2$ and $b=(d_n-1+n)/2$. By \eref{e:deranged},
  \begin{align*}
    \per\ev P_0+\per\ev(J-P_0)&\equiv1+(d_n+1-n)/2\equiv n/2\equiv
  (d_n-1+n)/2\\
    &\equiv\per\ev P_1+\per\ev(J-P_1)
  \end{align*}
  mod 2.
  Thus, in calculating \eref{e:evenpereq} we can pair up complementary terms
  in $R_1\ev$ and $R_{n-1}\ev$ to show that 
  $R_1\ev+R_{n-1}\ev\equiv n(n/2)\equiv0\mod2$. The result follows,
  since $\sum_i E_i\ev=n!/2\equiv0\mod2$.
\end{proof}

Note that $E_1\ev \equiv n - \pi_s \mod 2$, where $\pi_s$ is the
symbol parity described in \sref{s:intro}. It is curious that the
$E_1\ev$ term in \tref{tm:evper} appears on the side of the congruence
that it does.  The analogous statement for standard permanents follows
by considering \tref{t:oddevenE} modulo 2, and noting that the $E_1$
term can be written on either side of the congruence, since $E_1=n$ is
even.

We next show a parity relationship between consecutive pairs 
in the sequence $E_1,\dots,E_n$.

\begin{theorem}
If $L$ is a Latin square of even order
then $E_{2i-1} \equiv E_{2i} \bmod 2$ for each integer $i$.
\end{theorem}

\begin{proof}By \eref{eq:e-m},
\begin{align*}
E_{2i} + E_{2i-1} 
&= R_{2i} + \sum_{r=1}^{2i-1} \left[\binom{n-r}{n-2i} - \binom{n-r}{n-2i+1}\right] (-1)^r R_r\\
&= R_{2i} + \sum_{r=1}^{2i-1} \left[\binom{n-r}{n-2i} - \binom{n-r}{n-2i}\left(\frac{2i-r}{n-2i+1}\right)\right] (-1)^r R_r\\
&= R_{2i} + \sum_{r=1}^{2i-1} \binom{n-r}{n-2i} \left[\frac{n-4i+r+1}{n-2i+1}\right] (-1)^r R_r.
\end{align*}
If $r$ is even, then $R_r$ is even. If $r$ is odd, then $n-4i+r+1$ is even, while $n-2i+1$ is odd, so $\binom{n-r}{n-2i}({n-4i+r+1})/({n-2i+1})$ must be even (it is an integer, since our proof shows that it is the difference of two integers). The result follows.
\end{proof}

It seems that $E_{2i}$ and $E_{2i+1}$ are unrelated except for the case
$E_{n-2}$ and $E_{n-1}$ when $n\equiv 0 \bmod 4$, which is covered in
the following conjecture. 

\begin{conjecture}\label{conj:tij_Ek}
  Let $L$ be a Latin square of order $n$. The following holds for all
  $i,j$ and $r$.
  If $n \equiv 0 \bmod 4$, then $$E_n \equiv
  E_{n-1} \equiv 2E_{n-2} \equiv 2t_{ij} \equiv N_r \bmod 4,$$
  \begin{equation}\label{eq:sum-r_i}
    R_1 + R_3 + \cdots + R_{n-1} \equiv 0 \bmod 4 \qquad\text{and}\qquad
    R_2+R_4+\cdots+R_n \equiv E_n \bmod 4.
  \end{equation}
  If $n \equiv 2 \bmod 4$, then $E_{n-1} \equiv 2t_{ij} \equiv N_r \bmod 4$.
\end{conjecture}

Using \eref{eq:e-m}, \tref{th:bala} and \lref{lm:identities}(d),
if $n \equiv 0 \bmod 4$, then $\sum R_i \equiv E_n \bmod 4$.
Also, $E_{n-1} + 2E_{n-2}\equiv R_1 + R_3 + \cdots + R_{n-1}\bmod4$. Thus,
$E_{n-1} \equiv 2E_{n-2} \bmod 4$ is equivalent to \eref{eq:sum-r_i}. Note, 
for even $n$, that
$E_n$, $E_{n-1}$ and $N_r$ are even by \tref{th:bala}, \tref{th:E_n-1} and
\cyref{cor:n-r-even}, respectively.

Suppose that $n\=0\bmod4$ and that \cjref{conj:tij_Ek} holds. It
follows that $N_r+E_n\=0\bmod4$, which combines with
\lref{lm:delta-along-row} to imply that the number of transversals in
any $(n-1)\times n$ Latin rectangle is divisible by 4, whenever $n$
itself is divisible by 4.

Our results to this point have all been congruences mod 2 or 4. We
finish by showing for any given order $n$ that $R_2$ and $E_2$ have
only two possible values mod 6. The main interest in this result is 
that it involves a different modulus to our other results.

\begin{lemma}\label{l:mod3}
Let $L$ be a Latin square of any order $n$. Then $R_2\=E_2\=0\bmod2$ and
\[
R_2\nequiv(-1)^n(n+1)\bmod3\mbox{ \ and \ }
E_2\nequiv(-1)^n(n+1)-n(n-1)\bmod3.
\]
\end{lemma}

\begin{proof}
By \eref{eq:e-m} we have that $E_2=R_2-(n-1)R_1=R_2-n(n-1)$, so it
suffices to prove the claims about $R_2$. 
By \lref{lm:identities}(d), we know that $R_2$ is even, and it follows 
immediately that $E_2$ is even as well.

For each symbol $s$ of $L$,
define a permutation $\theta_s:[n]\rightarrow[n]$ by $\theta_s(i)=j$ if
$L_{ij}=s$.  Then $R_2$ is the sum over symbols $s,s'\in[n]$ of
$2^{c(s,s')}$, where $c(s,s')$ is the number of cycles in
$(\theta_s)^{-1}\theta_{s'}$. The number of cycles in any permutation
$\sigma\in\sym_n$ is $n-\epsilon(\sigma)\bmod2$.  Hence
\begin{align*}
R_2&=\sum_{s,s'}2^{c(s,s')}
\=\sum_{s,s'}(-1)^{n-\epsilon((\theta_s)^{-1}\theta_{s'})}
\=(-1)^n\sum_{s,s'}(-1)^{\epsilon(\theta_s)+\epsilon(\theta_{s'})}\\
&\=(-1)^n\left(\binom{e}{2}+\binom{n-e}{2}-e(n-e)\right)
\bmod3,
\end{align*}
where $e=\big|\{s\in[n]:\epsilon(\theta_s)=0\}\big|$.
The required result now follows by a simple case analysis concerning
the value of $e$ mod 3.
\end{proof}

\section{\label{s:conclusion}Concluding remarks}

We have shown a number of congruences satisfied by various quantities
motivated by the study of transversals in Latin squares.  There are many
others which are direct consequences of the results we have given. For
example, it is easy to use \eref{eq:e-m}, \lref{lm:identities}(d) and
\cyref{cor:odd-ls-mult-4} to show that $E_8\=0\bmod4$ when $n\=3\bmod4$,
given that
\begin{align*}
&\binom{n-7}{1}\=\binom{n-5}{3}\=\binom{n-3}{5}\=\binom{n-2}{6}
\=\binom{n-1}{7}\=0\bmod4\mbox{ \ and \ }\\
&\binom{n-6}{2}\=0\bmod2.
\end{align*}

As mentioned in the introduction, a notion of parity has been useful
in a number of different studies of Latin squares. In this paper we
have introduced parity for transversals of Latin squares, and used it
in the analysis of the number of transversals. In our investigation we
uncovered a number of interesting patterns, some of which we have
proved, and others we conjecture.  Several of the conjectures classify
Latin squares of a given even order into two types which seem to have
different properties. These classifications do not seem to be related
to each other, or to pre-existing notions of parity.  In
\cjref{conj:even} the value of $w$ mod $2$, say, partitions Latin
squares based on the parity of their transversals, while
\cjref{conj:tij_Ek} partitions Latin squares into two classes based on
$E_{n-1}$ modulo 4. However, these partitions seem to be independent
of each other and of the previously studied parities $\pi_r$ and
$\pi_c$. By randomly generating Latin squares, we found a Latin square
with each of the 16 possibilities for $(w,E_{n-1}/2,\pi_r,\pi_c)\bmod 2$ for
orders 8, 10 and 12.

Finally, we remark that we have only considered the classical
$2$-dimensional case in this paper. However, transversals are of
interest in the context of Latin hypercubes and permanents can also be
generalised to higher dimensions. All of the questions that we
have investigated could also be asked in these higher dimensional
contexts. A first step in this direction has been taken by Taranenko
\cite{Tar16}, who noted that \tref{th:bala} generalises to Latin
hypercubes.

\subsection*{Acknowledgement}
The authors are grateful to Saieed Akbari for interesting discussions on
the topic of this paper.

  \let\oldthebibliography=\thebibliography
  \let\endoldthebibliography=\endthebibliography
  \renewenvironment{thebibliography}[1]{%
    \begin{oldthebibliography}{#1}%
      \setlength{\parskip}{0.2ex}%
      \setlength{\itemsep}{0.2ex}%
  }%
  {%
    \end{oldthebibliography}%
  }


\begin{thebibliography}{99}

\bibitem{AL15}
R.~Aharoni and M.~Loebl, 
The odd case of Rota's bases conjecture,
{\em Adv. Math.} {\bf282} (2015), 427--442.

\bibitem{ABMW14}
R.\,E.\,L.~Aldred, R.\,A.~Bailey, B.\,D.~McKay and I.\,M.~Wanless,
Circular designs balanced for neighbours at distances one and two,
Biometrika {\bf101} (2014), 943--956.



\bibitem{AT92} N.~Alon and M.~Tarsi, 
Colorings and orientations of graphs, 
Combinatorica {\bf 12} (1992) 125--134.

\bibitem{Alp17}
L.~Alpoge,
Square-root cancellation for the signs of Latin squares,
{\it Combinatorica} {\bf 37} (2017), 137--142.

\bibitem{AA04}
S. Akbari and A. Alipour. 
Transversals and multicolored matchings. 
{\it J. Combin. Designs} {\bf12} (2004), 325--332.

\bibitem{Bal90}
K. Balasubramanian. 
On transversals in Latin squares, 
\textit{Linear Algebra Appl.} {\bf131} (1990), 125--129.


\bibitem{BN19}
J.~Bar\'at and Z.\,L.~Nagy,
Transversals in generalized Latin squares,
{\it Ars Math. Contemp.} {\bf16} (2019), 39--47.


\bibitem{BHWWW18}
D.~Best, K.~Hendrey, I.\,M.~Wanless, T.\,E.~Wilson and D.\,R.~Wood,
Transversals in Latin arrays with many distinct symbols,
{\it J.\ Combin.\ Des.} {\bf26} (2018), 84--96.


\bibitem{CW16}
N.\,J.~Cavenagh and I.\,M.~Wanless, 
There are asymptotically the same number of Latin squares of each parity,
{\it Bull. Aust. Math. Soc.} {\bf94} (2016), 187--194.

\bibitem{DGGL10} 
D.\,M.~Donovan, M.\,J.~Grannell, T.\,S.~Griggs and 
J.\,G.~Lefevre, On parity vectors of Latin squares, {\it Graphs Combin.} 
{\bf26} (2010) 673--684.

\bibitem{FHW18}
N.~Franceti\'c, S.~Herke and I.\,M.~Wanless, 
Parity of sets of mutually orthogonal Latin squares,
{\it J. Combin. Theory Ser. A\/} {\bf155} (2018), 67--99.

\bibitem{Gly10} D.\,G.~Glynn,
The conjectures of Alon-Tarsi and Rota in dimension prime minus one,
{\it SIAM J. Discrete Math.} {\bf24} (2010), 394--399. 

\bibitem{GB12} D.\,G.~Glynn and D.~Byatt, 
Graphs for orthogonal arrays and projective planes of even order,  
{\it SIAM J. Discrete Math.} {\bf 26} (2012), 1076--1087.

\bibitem{Jan95} J.\,C.\,M.~Janssen, 
On even and odd Latin squares, 
{\it J. Combin. Theory Ser. A}, {\bf 69} (1995) 173--181.

\bibitem{KMOW14}
P.~Kaski, A.\,D.\,S. Medeiros, P.\,R.\,J.~\"Osterg\aa rd and I.\,M.~Wanless, 
Switching in one-factorisations of complete graphs,
{\it Electron. J. Comb.\/} {\bf21}(2) (2014), \#P2.49.

\bibitem{KY20}
P.~Keevash and L.~Yepremyan,
On the number of symbols that forces a transversal,
{\it Comb. Prob. Comput.}, to appear. DOI 10.1017/S0963548319000282



\bibitem{Kot12}
D.~Kotlar,
Parity types, cycle structures and autotopisms of Latin squares,
{\em Electron. J. Combin.} {\bf19}(3) (2012), \#P10.


\bibitem{MMW06}
B.\,D.~McKay, J.\,C.~McLeod and I.\,M.~Wanless, 
The number of transversals in a Latin square, 
{\it Des.\ Codes Cryptogr.} {\bf40} (2006), 269--284.

\bibitem{MPS19} 
R. Montgomery, A. Pokrovskiy and B. Sudakov,
Decompositions into spanning rainbow structures,
\emph{Proc. Lond. Math. Soc.} {\bf119} (2019), 899--959.


\bibitem{New78}
M. Newman,
Combinatorial matrices with small determinants, 
{\it Canad. J. Math.} {\bf30} (1978), 756--762.

\bibitem{Rys63}
H.\,J. Ryser, Combinatorial Mathematics, 
\textit{The Carus Mathematical Monographs} {\bf14}, 
The Mathematical Association of America, 1963.

\bibitem{ryser}
H.\,J. Ryser, Neuere Probleme der Kombinatorik, 
\textit{Vortrage \"uber Kombinatorik Oberwolfach}, 24-29 Juli (1967), 69--91.

\bibitem{SW12}
D.\,S.~Stones and I.\,M.~Wanless, 
How not to prove the Alon-Tarsi conjecture,
{\it Nagoya Math.\ J.} {\bf205} (2012), 1--24.

\bibitem{Tar16}
A.\,A.~Taranenko, 
Permanents of multidimensional matrices: properties and applications, 
{\em J. Appl. Ind. Math.} {\bf10} (2016), 567--604.

\bibitem{Wan04} I.\,M.~Wanless, 
Cycle switches in Latin squares, 
{\it Graphs Combin.} {\bf20} (2004) 545--570.


\bibitem{trans-survey}
I.\,M.~Wanless,
Transversals in Latin squares: a survey. 
In \textit{Surveys in combinatorics 2011}, 
London Math. Soc. Lecture Note Ser. {\bf392}, 403--437. 
Cambridge Univ. Press, Cambridge, 2011.


\end{thebibliography}
\end{document}